\documentclass[article,11pt]{amsart}
\usepackage[utf8]{inputenc}

\usepackage{graphicx, amsmath, amssymb, amscd, amsthm, euscript, psfrag, amsfonts,bm}

\usepackage{enumitem}
\usepackage{etoolbox}
\usepackage{amsmath}
\usepackage{enumerate}
\usepackage{amssymb}
\usepackage{amscd}
\usepackage{amsthm}
\usepackage{amsfonts}
\usepackage{graphicx}
\usepackage{tensor}
\usepackage{multirow}
\usepackage{multicol}
\usepackage{mathscinet}
\usepackage[T1]{fontenc}
\usepackage{mathrsfs}

\usepackage[title]{appendix}%
\usepackage{xcolor}%
\usepackage{textcomp}%
\usepackage{manyfoot}%
\usepackage{booktabs}%
\usepackage{algorithm}%
\usepackage{algorithmicx}%
\usepackage{algpseudocode}%
\usepackage{listings}%

\usepackage{soul}

\usepackage[colorlinks=true]{hyperref}
\hypersetup{urlcolor=blue, citecolor=red}
\usepackage{cleveref}
\crefformat{section}{\S#2#1#3} 
\crefformat{subsection}{\S#2#1#3}
\crefformat{subsubsection}{\S#2#1#3}

\newtheorem{theorem}{Theorem}[section]
\newtheorem{corollary}[theorem]{Corollary}

\newtheorem{lemma}[theorem]{Lemma}
\newtheorem{claim}{Claim}[theorem]
\newtheorem{proposition}[theorem]{Proposition}

\theoremstyle{definition}
\newtheorem{definition}[theorem]{Definition}
\newtheorem{example}[theorem]{Example}
\newtheorem{remark}[theorem]{Remark}

\title{On connectedness in the parametric geometry of numbers}

\author{Yuming Wei}
\address{Yau Mathematical Sciences Center, Tsinghua University, Beijing, 100084, China }
\email{weiym20@mails.tsinghua.edu.cn}
\author{Han Zhang}
\address{School of Mathematical Science, Soochow University, Suzhou 215006, China
}
\email{hanzhang2013@outlook.com}

\date{}

\begin{document}

\begin{abstract}
    Via multilinear algebra, we formulate a criterion for connectedness in the parametric geometry of numbers in terms of pencils, which are certain algebraic varieties in the space of matrices. As a consequence, we obtain a connectedness result for generic lattices arising from Diophantine approximation on analytic submanifolds, and sharpen Schmidt and Summerer's results of connectedness on simultaneous Diophantine approximation and approximation by linear forms.

\end{abstract}




\maketitle

\section{Introduction}\label{Section: Introduction}

\subsection{Background}
Let $d\geq 2$ be an integer. It is well known that the space of unimodular lattices in $\mathbb{R}^d$ can be identified with $\rm{SL}_d(\mathbb{R})/\rm{SL}_d(\mathbb{Z})$. Let $m,n\geq 1$ be integers such that $m+n=d$. We denote by $M_{m,n}(\mathbb{R})$ the set of $m\times n$ real matrices. Because in this article we only consider real matrices, we write $M_{m,n}(\mathbb{R})=M_{m,n}$ for simplicity. Given $x\in M_{m,n}$, we consider the upper triangular matrix 
\begin{align}\label{align: definition of u(x)}
    u(x)=\begin{pmatrix}
        \mathbb{I}_m & x \\
        0  & \mathbb{I}_n
\end{pmatrix},
\end{align}
where $\mathbb{I}_m$ ($\mathbb{I}_n$, resp.) is the $m\times m$ ($n\times n$, resp.) identity matrix. Let $\Lambda_x$ be the unimodular lattice defined by 
\[\Lambda_x=u(x)\mathbb{Z}^d.\]

\begin{definition}
    A vector $\omega=(\omega_1,\cdots,\omega_d)\in \mathbb{R}^d$ is called an $(m,n)$-\textit{weight vector} if $\omega_1\geq \cdots\geq \omega_m>0$, $\omega_{m+1}\geq \cdots\geq \omega_d>0$ and $\sum_{i=1}^m \omega_i=\sum_{j=m+1}^d \omega_{j}=1$. {We say} $\omega$ is of \textit{equal weight} if $\omega_1=\cdots=\omega_{m}=1/m$ and $\omega_{m+1}=\cdots=\omega_d=1/n$.
\end{definition}
Given an $(m,n)$-weight vector $\omega$, we define a one-parameter diagonal flow $\{a^{\omega}_t:t\in \mathbb{R}\}\subset \rm{SL}_d(\mathbb{R})$ by
\begin{align}\label{align: definition of a_t}
    a^{\omega}_t=\rm{diag} (e^{\omega_1 t},\cdots, e^{\omega_m t}, e^{-\omega_{m+1} t},\cdots, e^{- \omega_{d} t}).
\end{align}
For $1\leq i\leq d$, the $i$th successive minimum $\lambda_i(a^{\omega}_t \Lambda_x)$ is the least number $\lambda>0$ for which $\lambda \mathcal{C}$ contains $i$ linearly independent vectors in $a^{\omega}_t \Lambda_x$, where $\mathcal{C}$ is the closed unit cube of $\mathbb{R}^d$ centered at $0$. Since the discovery of Dani's correspondence \cite{Dani_1985_Divergent_trajectories_of_flows_on_homogeneous_spaces_MR794799}, {the first of the successive minima} of the forward trajectory $\{a^{\omega}_t \Lambda_x:t\geq 0\}$ was intensively studied from the dynamical point of view, while the other successive minima were less studied.

In two important papers of Schmidt and Summerer \cite{Schmidt_Summerer_2009_Parametric_geometry_of_numbers_and_applications_MR2557854,Schmidt_Summerer_2013_Diophantine_approximation_and_parametric_geometry_of_numbers_MR3016519}, they proposed the study of {the} parametric geometry of numbers, that is, one should investigate all successive minima simultaneously. Using this new theory, they sharpened many classical results in Diophantine approximation. Schmidt and Summerer's idea was developed in a series of papers (see e.g. \cite{Bugeaud_Laurent_2010_On+transfer+inequalities+in+Diophantine+approximation+II_MR2609309,German_2012_On+Diophantine+exponents+and+Khintchine's+transference+principle_MR2988525,Moshchevitin_2012_Exponents+for+three+dimensional+simultaneous+Diophantine+approximations_MR2899740}), and finally culminated in the fundamental work of Roy \cite{Roy_2015_On_Schmidt_and_Summerer_parametric_goemMR3418530,Roy_2016_Spectrum_of_the_exponents_of_best_rational_approximation_MR3489062}. Roy's result was further developed in \cite{Das_Fishman_Simmons_Urbanski_2024_A+variational+principle_MR4671568,Solan_2021_Parametric+geometry+of+numbers+with+general+flow}.

In order to study successive minima, given a unimodular lattice $\Lambda$ and an $(m,n)$-weight vector $\omega$, it will be convenient for us to define functions
\begin{align*}
    &L^{\omega}_{i}(\Lambda,t)=\log \lambda_i (a^{\omega}_t \Lambda), \quad \forall 1\leq i\leq d, \text{ and }\\
    & \boldsymbol{L}^{\omega}(\Lambda,t)=(L^{\omega}_{1}(\Lambda,t),\cdots, L^{\omega}_{d}(\Lambda,t)):\mathbb{R}_{\geq 0} \to \mathbb{R}^d.
\end{align*}
By the definition of $a^{\omega}_t$, each $L^{\omega}_{i}(\Lambda,t)$ is a piecewise linear function with slope lying in $\{\omega_1,\cdots,\omega_m,-\omega_{m+1},\cdots,-\omega_d\}$. One particular interest in the parametric geometry of numbers is to investigate the following connectedness property of $ \boldsymbol{L}^{\omega}(\Lambda,\cdot)$:
\begin{definition}\label{Definition: connectedness}
Given $k\in \{1,\cdots,d-1\}$ and a unimodular lattice $\Lambda$, $\boldsymbol{L}^{\omega}(\Lambda,\cdot)$ is {connected} at $k$ if the set
\begin{align*}\label{align: the set of t where two successive minimas are equal}
    \{t\geq 0 : L^{\omega}_{k}(\Lambda,t)=L^{\omega}_{k+1}(\Lambda,t)\}
\end{align*}
is unbounded above. $\boldsymbol{L}^{\omega}(\Lambda,\cdot)$ is connected if it is connected at every $1\leq k\leq d-1$.  $\boldsymbol{L}^{\omega}(\Lambda,\cdot)$ is disconnected at $k$ if $L^{\omega}_{k}(\Lambda,t)<L^{\omega}_{k+1}(\Lambda,t)$ for all large $t$.
\end{definition}

\begin{remark}
    In an appropriate sense, connectedness measures rationality of a lattice. For example, let $d=2$, $x\in \mathbb{R}$ and $\omega=(1,1)$. Then $\boldsymbol{L}^{\omega}(\Lambda_{x},\cdot)$ is connected at $k=1$ if and only if $x\notin \mathbb{Q}$.
\end{remark}
Given a unimodular lattice $\Lambda$ and $1\leq k\leq d-1$, it is important to understand whether $\boldsymbol{L}^{\omega}(\Lambda,\cdot)$ is connected at $k$. As noted by Schmidt \cite[Proposition 4.2]{Schmidt_2020_On_parametric_geometry_of_numbers_MR4121876}, if $\boldsymbol{L}^{\omega}(\Lambda,\cdot)$ is disconnected at $k$, then the investigation of the parametric geometry of numbers with respect to $\Lambda$ breaks into two lower dimensional complementary sublattices $\Lambda^1$ and $\Lambda^2$ generating $\Lambda$.

In \cite{Schmidt_Summerer_2009_Parametric_geometry_of_numbers_and_applications_MR2557854}, a criterion for connectedness of $\boldsymbol{L}^{\omega}(\Lambda,\cdot)$ is given. To introduce this criterion, let us prepare some notations. For any $I\subset \{1,\cdots, d\}$, we associate to it the number
\[\mu_I=\sum_{i\in I,i\leq m} \omega_i-\sum_{i\in I, i\geq m+1}\omega_i.\]
Denote by $\pi_I:\mathbb{R}^d\to \mathbb{R}^I$ the orthogonal projection, where $\mathbb{R}^I$ is the linear subspace spanned by $\{e_i:i\in I\}$ and $e_i$ is the $i$th standard basis vector of $\mathbb{R}^d$. All vectors in this article are considered to be column vectors unless specified. For a column vector $w\in \mathbb{R}^d$, we denote $^t w$ to be its transpose. The following is the criterion:

\begin{theorem}\label{Theorem: geometric connectedness}\cite[Theorem 1.1]{Schmidt_Summerer_2009_Parametric_geometry_of_numbers_and_applications_MR2557854}
Given an $(m,n)$-weight vector $\omega$. Let $k\in \{1,\cdots, d-1\}$ and $\Lambda$ be a unimodular lattice. Suppose for every $k$-dimensional linear space $W\subset \mathbb{R}^d$ spanned by vectors of $\Lambda$, there is some $I\subset \{1,\cdots,d\}$ of cardinality $k$ such that $\mu_I>0$\footnote{In \cite[Theorem 1.1]{Schmidt_Summerer_2009_Parametric_geometry_of_numbers_and_applications_MR2557854}, the condition is $\mu_I<0$, this is because there $\boldsymbol{L}$ is defined with respect to $\Lambda$ and $a^{\omega}_{-t}\mathcal{C}$.} and $\pi_I(W)=\mathbb{R}^I$. Then $\boldsymbol{L}^{\omega}(\Lambda,\cdot)$ is connected at $k$.
\end{theorem}
Theorem \ref{Theorem: geometric connectedness} is stated in a geometric way. However, it turns out that it is helpful to reformulate this theorem in terms of multilinear algebra. The main purpose of this article is to make use of techniques in multilinear algebra to obtain: (1) the relation between connectedness in the parametric geometry of numbers and certain algebraic varieties (i.e. pencils) in $M_{m,n}$; (2) a connectedness result of generic lattices associated to an analytic submanifold of $M_{m,n}$; (3) a sharpening of connectedness result of Schmidt and Summerer on simultaneous Diophantine approximation and approximation by linear forms.

\begin{remark}
By constructing an explicit example, Summerer \cite[Corollary 4.2]{Summerer_2019_Generalized+simultaneous+approximation_MR3990805} showed that the condition in Theorem \ref{Theorem: geometric connectedness} is not necessary for $\boldsymbol{L}^{\omega}(\Lambda,\cdot)$ to be connected at certain $k$. This indicates subtlety of connectedness property.

\end{remark}

\subsection{Main results}
 In what follows, we first reformulate Theorem \ref{Theorem: geometric connectedness} in the language of multilinear algebra. Given $k\in \{1,\cdots,d-1\}$, denote by $\bigwedge^k \mathbb{R}^d$ the $k$th exterior product of $\mathbb{R}^d$. Then $a^{\omega}_t$ naturally acts on $\bigwedge^k \mathbb{R}^d$. Since $a^{\omega}_t$ is diagonalizable, any vector $v\in \bigwedge^k \mathbb{R}^d$ can be decomposed into the sum of eigenvectors of $a^{\omega}_t$.
For any $k$-dimensional linear subspace $W\subset \mathbb{R}^d$ such that $W$ is spanned by vectors $w_1,\cdots,w_k$, let $v_W=w_1\wedge\cdots\wedge w_k$ represent $W$ in $\bigwedge^k \mathbb{R}^d$. We note that this representation is unique up to a scalar multiple.

\begin{theorem}\label{Theorem: algebraic connectedness}
Given an $(m,n)$-weight vector $\omega$, let $k\in \{1,\cdots,d-1\}$ and $\Lambda$ be a unimodular lattice. Suppose for every $k$-dimensional space $W$ spanned by vectors of $\Lambda$, the largest eigenvalue in the decomposition of $v_W$ into eigenvectors of $a^{\omega}_1$ in $\Lambda^k \mathbb{R}^d$ is strictly greater than $1$. Then $\boldsymbol{L}^{\omega}(\Lambda,\cdot)$ is connected at $k$.
\end{theorem}
We will give a proof of Theorem \ref{Theorem: algebraic connectedness} in Section \ref{Section: proof of main theorem and corollary}, which is algebraic and similar to that of Schmidt and Summerer. It is not hard to see that
\begin{proposition}\label{Proposition: Two theorems are equivalent}
    Theorem \ref{Theorem: geometric connectedness} and Theorem \ref{Theorem: algebraic connectedness} are equivalent.
\end{proposition}
For completeness, a proof of this equivalence will be given in Section \ref{Section: Preliminaries}.
Although Theorem \ref{Theorem: algebraic connectedness} is a restatement of Theorem \ref{Theorem: geometric connectedness}, it allows us to analyze connectedness in the parametric geometry of numbers through pencils, which are certain algebraic varieties in $M_{m,n}$ and closely related to Schubert varieties in the grassmannian. The relation between pencils and diophantine approximation was revealed in \cite{Aka_Breuillard_Rosenzweig_Saxce_2018_Diophantine+approximation_MR3777412}. We give a brief introduction to pencils in the following and refer the interested reader to \cite{Aka_Breuillard_Rosenzweig_Saxce_2018_Diophantine+approximation_MR3777412} and the references therein for more details.

To begin with, for any $x\in M_{m,n}$, we denote by $\tilde{x}\in M_{m,d}$ the first $m$ rows of $u(x)$. For linear subspaces $W\subset \mathbb{R}^d$ and $W'\subset \langle e_1,\cdots, e_m \rangle \cong \mathbb{R}^m$, we consider the following subsets of $\{1,\cdots,d\}$:
\begin{align*}
    I(W)=\{i\in \{m+1,\cdots,d\} : \dim (W\cap \langle e_1,\cdots, e_i \rangle )>\dim (W\cap \langle e_1,\cdots,e_{i-1} \rangle)\};\\
    J(W')=\{j\in \{1,\cdots,m\} : \dim (W'\cap \langle e_{j+1},\cdots, e_m \rangle)<\dim (W' \cap \langle e_j,\cdots, e_m \rangle)\},
\end{align*}
where we interpret $\langle e_{j+1},\cdots,e_m \rangle=0$ if $j=m$. We define the following functions
\begin{align}\label{align: definition of psi and phi}
    \psi(W)=\sum_{i\in I(W)} \omega_i, \quad \phi(W')=\sum_{j\in J(W')} \omega_j.
\end{align}
By convention, we set $\psi(W)=0$ ($\phi(W')=0$, resp.) if $I(W)=\emptyset$ ($J(W')=\emptyset$, resp.).
\begin{remark}
Let $x\in M_{m,n}$ and $W$ be a linear subspace of $\mathbb{R}^d$.
Roughly speaking, $J(\Tilde{x}W)$ ($I(\ker \tilde{x}\cap W)$, resp.) records a set of basis vectors of $W\cap (\ker \Tilde{x})^{\perp}$ ($\ker \Tilde{x}\cap W$, resp.) satisfying the following property: each basis vector multiplied by $u(x)$ on the left has largest possible expanding rate (smallest possible contracting rate, resp.) under the action of $a^{\omega}_1$. For example, when $d=4$, $m=n=2$ and $x=0$. Let $w_1={^t(} 1,1,0,1), w_2={^t(}0,0,1,1)$ and $W=\langle w_1,w_2 \rangle$. Then $J(\Tilde{x}W)=\{1\},I(\ker\tilde{x}\cap W)=\{4\}$, and $\phi(\Tilde{x}W)=\omega_1,\psi(\ker\Tilde{x}\cap W)=\omega_4$.

When $\omega$ is of equal weight, $\phi(\tilde{x}W)= 1/m \cdot \dim \tilde{x}W$ and $\psi(\ker \Tilde{x}\cap W)=1/n \cdot\dim \ker \Tilde{x}\cap W$.
\end{remark}

\begin{definition}\label{Definition: pencils}
 Given an $(m,n)$-weight vector $\omega$ and $k\in\{1,\cdots,d-1\}$. Let $a,b\in \mathbb{R}$ and $W$ be a $k$-dimensional linear subspace of $\mathbb{R}^d$. A {$k$-pencil}\footnote{This definition of pencil is an adaption of that in \cite{Aka_Breuillard_Rosenzweig_Saxce_2018_Diophantine+approximation_MR3777412}.} $\mathcal{P}^{\omega}_{W,a,b}$ is defined by
  \begin{align*}
      \mathcal{P}^{\omega}_{W,a,b}=\{x\in M_{m, n}: \psi(\ker \tilde{x} \cap W)\geq a \text{ and } \phi(\Tilde{x}W)\leq b\}.
  \end{align*}
If $W$ is a $k$-dimensional linear subspace of $\mathbb{R}^d$ defined over $\mathbb{Q}$ (i.e., $W$ has a $\mathbb{Q}$-basis), then $\mathcal{P}^{\omega}_{W,a,b}$ is called a {rational} $k$-pencil. If $b\leq a$, then $\mathcal{P}^{\omega}_{W,a,b}$ is called a {weakly constraining} $k$-pencil. If $\mathcal{P}^{\omega}_{W,a,b}\neq M_{m,n}$, then $\mathcal{P}^{\omega}_{W,a,b}$ is {proper}.
\end{definition}

\begin{remark}
In Proposition \ref{Proposition: pencils are Zariski closed}, we will show that given an $(m,n)$-weight $\omega$ and $k\in \{1,\cdots,d-1\}$, any $k$-pencil is a Zariski closed algebraic variety (not necessarily irreducible) of $M_{m,n}$.

\end{remark}

We have the following result:
\begin{theorem}\label{Theorem: generic behavior of connectedness}
   Given an $(m,n)$-weight vector $\omega$ and $k\in \{1,\cdots,d-1\}$. Let $\mathcal{M}$ be a connected analytic submanifold of $M_{m,n}$. Assume that ${\mathcal{M}}$ is not contained in any proper rational weakly constraining $k$-pencil. Then for Lebesgue\footnote{Here by Lebesgue measure on an analytic manifold, we mean the natural volume measure.} almost every $x\in \mathcal{M}$, $\boldsymbol{L}^{\omega}(\Lambda_x,\cdot)$ is connected at $k$.
\end{theorem}
In particular, when $m=1$ we have the following:

\begin{corollary}\label{Corollary: a manifold not contained in affine plane is connected}
Given a $(1,d-1)$-weight vector $\omega$. Let $\mathcal{M}$ be a connected analytic submanifold of $M_{1,d-1}\cong \mathbb{R}^{d-1}$. Assume that ${\mathcal{M}}$ is not contained in any proper affine hyperplane of $\mathbb{R}^{d}$, then for Lebesgue almost every $x\in \mathcal{M}$, $\boldsymbol{L}^{\omega}(\Lambda_x,\cdot)$ is connected.
\end{corollary}

\begin{example}
    Let $\omega$ be a $(1,d-1)$-weight vector and $\mathcal{M}$ be the Veronese curve defined by
 \[\mathcal{M}=\{(s,\cdots,s^{d-1}): s\in \mathbb{R}\}\subset M_{1,d-1}.\]
It is not hard to see that ${\mathcal{M}}$ is not contained in any proper affine hyperplane of $\mathbb{R}^{d-1}$. By Corollary \ref{Corollary: a manifold not contained in affine plane is connected}, for Lebesgue almost every $x\in \mathcal{M}$, $\boldsymbol{L}^{\omega}(\Lambda_x,\cdot)$ is connected.
\end{example}

When the manifold $\mathcal{M}=\{x\}$ is a singleton, a sufficient condition for $x$ ensuring the connectedness of $\boldsymbol{L}^{\omega}(\Lambda_x,\cdot)$ can be given in terms of coordinates of $x$. In \cite{Schmidt_Summerer_2009_Parametric_geometry_of_numbers_and_applications_MR2557854}, Schmidt and Summerer showed the following: Let $x={^t(}\xi_1,\cdots,\xi_{d-1})\in M_{d-1,1}$ ($y=(\xi_1,\cdots,\xi_{d-1})\in M_{1,d-1}$, resp.). Assume $1,\xi_1,\cdots,\xi_{d-1}$ are linearly independent over $\mathbb{Q}$. Then given any $(d-1,1)$-weight vector $\omega$ ( $(1,d-1)$-weight vector $\omega'$, resp. ), $\{x\}$ ($\{y\}$, resp.) is not contained in any proper rational weakly constraining $k$-pencil for every $1\leq k\leq d-1$, and thus is connected at every $k$. By employing techniques in multilinear algebra, we obtain the following corollaries sharpening these results:

\begin{corollary}\label{Corollary: connectedness in 1 by d}
Given a {$(d-1,1)$}-weight vector $\omega$, $x={^t(}\xi_1,\cdots,\xi_{d-1})\in M_{d-1,1}$, and $k\in \{1,\cdots,d-1\}$. Assume there exists $\{n_1,\cdots,n_k\}\subset \{1,\cdots,d-1\}$ such that $1,\xi_{n_1},\cdots,\xi_{n_k}$ are linearly independent over $\mathbb{Q}$\footnote{{This is equivalent to saying that ${^t(}1,\xi_1,\cdots,\xi_{d-1})$ is not contained in any $k$-dimensional rational subspace $W$}.}. Then $\{x\}$ is not contained in any proper rational weakly constraining $k$-pencil, and $\boldsymbol{L}^{\omega}(\Lambda_x,\cdot)$ is connected at $k$. In particular, if $1,\xi_1,\cdots,\xi_{d-1}$ are linearly independent over $\mathbb{Q}$, then
$\boldsymbol{L}^{\omega}(\Lambda_{x},\cdot)$ is connected at every $1\leq k\leq d-1$.
\end{corollary}

\begin{corollary}\label{Corollary: connectedness in the dual version}
Given a {$(1,d-1)$}-weight vector $\omega$, $y=(\xi_1,\cdots,\xi_{d-1})\in M_{1,d-1}$, and $k\in \{1,\cdots,d-1\}$. Assume there exists $\{n_1,\cdots,n_k\}\subset \{1,\cdots,d-1\}$ such that $1,\xi_{n_1},\cdots,\xi_{n_k}$ are linearly independent over $\mathbb{Q}$. Then $\{y\}$ is not contained in any proper rational weakly constraining $(d-k)$-pencil, and $\boldsymbol{L}^{\omega}(\Lambda_y,\cdot)$ connected at $d-k$. In particular, if $1,\xi_1,\cdots,\xi_{d-1}$ are linearly independent over $\mathbb{Q}$, then $\boldsymbol{L}^{\omega}(\Lambda_{y},\cdot)$ is connected at every $1\leq k\leq d-1$.
\end{corollary}

\begin{remark}
The upshot of Corollaries \ref{Corollary: connectedness in 1 by d} and \ref{Corollary: connectedness in the dual version} is that in contrast to requiring all coordinates of $x$ ($y$, resp.) with $1$ to be linearly independent over $\mathbb{Q}$ as in \cite{Schmidt_Summerer_2009_Parametric_geometry_of_numbers_and_applications_MR2557854}, we only need $k$ many coordinates of $x$ ($d-k$ many coordinates of $y$, resp.) with $1$ to be linearly independent over $\mathbb{Q}$ to ensure that $\boldsymbol{L}^{\omega}(\Lambda_{x},\cdot)$ ($\boldsymbol{L}^{\omega}(\Lambda_{y},\cdot)$, resp.) is connected at $k$. 
\end{remark}

Given general positive integers $m,n$ with $m+n=d$ and an $(m,n)$-weight vector $\omega$. In \cite[Proposition 9.1]{Schmidt_2020_On_parametric_geometry_of_numbers_MR4121876}, Schmidt studied conditions for connectedness of $\boldsymbol{L}^{\omega}(\Lambda_x,\cdot)$ for $x\in M_{m,n}$. We may write $x=(x_{i,j})_{1\leq i\leq m,1\leq j\leq n}$. For any $k\in \{1,\cdots,d-1\}$ and $\mathcal{R}\subset \{1,\cdots,n\}$ of cardinality $d-k$, we denote 
\begin{align*}
    \mathcal{D}(\mathcal{R})=\{\det(x_{i_s,j_t})_{1\leq s,t\leq p}: 1\leq p\leq \min\{m,d-k\}, & 1\leq i_1<\cdots<i_p\leq m,\\ & j_1<\cdots<j_p\in \mathcal{R}  \}.
\end{align*}

\begin{proposition}\label{Proposition: matrix connectedness}\cite[Proposition 9.1]{Schmidt_2020_On_parametric_geometry_of_numbers_MR4121876}

(a) Given $k\in \{m,\cdots,d-1\}$. Suppose that for each $\mathcal{R}$ as above of cardinality $d-k$, the elements of $\mathcal{D}(\mathcal{R})$ with $1$ are linearly independent over $\mathbb{Q}$. Then $\boldsymbol{L}^{\omega}(\Lambda_x,\cdot)$ is connected at $k$.

(b) Given $k\in\{1,\cdots,m\}$. Suppose for $\mathcal{R}=\{1,\cdots,n\}$, the elements of $\mathcal{D}(\mathcal{R})$ with $1$ are linearly independent over $\mathbb{Q}$. Then $\boldsymbol{L}^{\omega}(\Lambda_x,\cdot)$ is connected at $k$.
\end{proposition}
We will give an alternative proof of this proposition using multilinear algebra.

\section{Preliminaries}\label{Section: Preliminaries}
Let $d\geq 2$ be an integer. For $i\in \{1,\cdots,d\}$, denote by ${e}_i$ the standard basis vector of $\mathbb{R}^d$, i.e., ${e}_i$ is a column vector with $1$ in the $i$-th row and $0$ elsewhere. Here and in the sequel, we fix $k\in\{1,\cdots,d-1\}$ and a weight vector $\omega$.

\subsection{Multilinear algebra}
For a subset $I=\{i_1,\cdots,i_k\}\subset \{1,\cdots,d\}$ with cardinality $\# I=k$, we often write $I=\{i_1<\cdots<i_k\}$ to emphasize the order of elements in $I$. We denote by $\pi_I:\mathbb{R}^d\to \mathbb{R}^I$ the orthogonal projection, where $\mathbb{R}^I$ is the $k$-dimensional linear subspace spanned by $\{e_i:i\in I\}$. Denote by $\bigwedge^k \mathbb{R}^d$ the $k$th exterior algebra of $\mathbb{R}^d$. We equip $\bigwedge^k \mathbb{R}^d$ with basis $\{e_I: \#I =k\}$, where $e_I=e_{i_1}\wedge\cdots\wedge e_{i_k}$ when $I=\{i_1<\cdots<i_k\}$. Note that any $g\in \rm{SL}_d(\mathbb{R})$ naturally acts on $\bigwedge^k \mathbb{R}^d$ by 
 \[g( v_1\wedge\cdots\wedge v_k)=gv_1\wedge\cdots\wedge gv_k.\]
Since $a^{\omega}_1$ is diagonal, any vector in $\bigwedge^k \mathbb{R}^d$ can be decomposed into the sum of eigenvectors of $a^{\omega}_1$. By definition of $a^{\omega}_1$, it turns out that eigenvectors of $a^{\omega}_1$ are exactly $e_I$'s, where $\# I=k$. Moreover, $a^{\omega}_1 e_I=\exp(\mu_I) e_I$ for $\mu_I=\sum_{i\in I,i\leq m} \omega_i-\sum_{i\in I, i\geq m+1}\omega_i$.

\begin{lemma}\label{Lemma: projection is onto iff component is nonzero}
    Let $k\in \{1,\cdots,d-1\}$ and $W\subset \mathbb{R}^d$ be a $k$-dimensional linear subspace spanned by vectors $\{w_1,\cdots,w_k\}$. Let $v_W=w_1\wedge\cdots \wedge w_k\in \bigwedge^k \mathbb{R}^d$ represent $W$. Then $\pi_I(W)=\mathbb{R}^I$ if and only if the $e_I$-component in $v_W$ is nonzero.
\end{lemma}

\begin{proof}
For each vector $w_i$, we may write
\[w_i=\pi_I(w_i)+\pi_{I^c}(w_i),\]
where $I^c=\{1,\cdots,d\}\setminus I$. Therefore,
\begin{align*}
    v_W&=(\pi_I(w_1)+\pi_{I^c}(w_1))\wedge\cdots\wedge(\pi_I(w_k)+\pi_{I^c}(w_k))\\
    &=c_I e_I + \sum_{J\neq I, \# J=k} c_J e_J,
\end{align*}
where $\pi_I(w_1)\wedge\cdots \wedge\pi_I(w_k)=c_I e_I$. Assume that $\pi_I(W)=\mathbb{R}^I$. 
 Then since $\dim W=k$ and $(w_1,\cdots,w_k)$ is a basis of $W$, $(\pi_I(w_1),\cdots,\pi_I(w_k))$ is a basis of $\mathbb{R}^I$. Hence, $c_I\neq 0$.

Conversely, assume that $c_I\neq 0$. Then $\pi_I(w_1)\wedge\cdots \wedge\pi_I(w_k)\neq 0$. Therefore, $(\pi_I(w_1),\cdots,\pi_I(w_k))$ form a basis of $\mathbb{R}^I$. That is, $\pi_I(W)=\mathbb{R}^I$.
\end{proof}

\begin{proof}[Proof of equivalence between Theorem \ref{Theorem: geometric connectedness} and Theorem \ref{Theorem: algebraic connectedness}]
     Suppose for a $k$-dimensional linear space $W$ spanned by vectors of $\Lambda$, there is some $I$ of cardinality $k$ such that $\mu_I>0$ and $\pi_I(W)=\mathbb{R}^I$. Then by Lemma \ref{Lemma: projection is onto iff component is nonzero} the coefficient of $e_I$ in $v_W$ is nonzero. As $a^{\omega}_1 e_I= \exp({\mu_I}) e_I$, $W$ satisfies the assumption of Theorem \ref{Theorem: algebraic connectedness}.
     
     Conversely, if $W$ satisfies the assumption of Theorem \ref{Theorem: algebraic connectedness}, then there exists $I$ of cardinality $k$ with $\mu_I>0$ such that $e_I$-component of $v_W$ is nonzero. Again by Lemma \ref{Lemma: projection is onto iff component is nonzero}, $\pi_I(W)=\mathbb{R}^I$.
\end{proof}

The following is an observation made in the proof of \cite[Theorem 5.17]{Aka_Breuillard_Rosenzweig_Saxce_2018_Diophantine+approximation_MR3777412}. For completeness, we include a proof as follows.
\begin{lemma}\label{Lemma: highest eigenvalue determined by pencil}
    Given an $(m,n)$-weight vector $\omega$, $x\in M_{m, n}$ and a $k$-dimensional linear subspace $W\subset \mathbb{R}^d$. Let $v_W$ represent $W$ in $\bigwedge^k \mathbb{R}^d$. Then the largest eigenvalue in the decomposition of $u(x)v_W$ into eigenvectors of $a^{\omega}_1$ in $\bigwedge^k \mathbb{R}^d$ is $\exp(\phi(\Tilde{x}W)-\psi(\ker \Tilde{x}\cap W))$, where $\phi,\psi$ are defined as in (\ref{align: definition of psi and phi}).
\end{lemma}

\begin{proof}
Let $s,r\geq 0$ be two integers such that $s+r=k$, $\dim \Tilde{x}W=r$ and $\dim \ker\Tilde{x}\cap W =s$. By definition, for any $i\in I(\ker \Tilde{x}\cap W)$, there exists $w_i\in \ker \Tilde{x}\cap W$ such that
    \[w_i=e_i+\sum_{j<i} c_{ij}e_j.\]
Likewise, for any $i\in J(\Tilde{x}W)$, there exists $w_i\in W$ such that
    \[\Tilde{x}w_i= e_i+\sum_{i<j\leq m} c_{ij} e_j.\]
Denoting $K= I(\ker \tilde{x}\cap W)\cup J(\Tilde{x}W)$, we may choose $(w_i)_{i\in K}$ to form a basis of $W$. We observe that 
\begin{align*}
    &u(x)w_i=e_i + \sum_{m+1\leq j<i} c_{ij}e_j, \quad \forall i\in I(\ker\Tilde{x}\cap W);\\
    &u(x)w_i=e_i+\sum_{j>i} c_{ij}e_j, \quad \forall i \in J(\tilde{x}W),
\end{align*}
where the second equality follows as we write $\sum_{j\geq m+1}c_{ij}e_j=\pi_{\{m+1,\cdots,d\}}w_i$ for $i\in J(\Tilde{x}W)$. Hence,
\begin{align*}
    \bigwedge_{i\in  K} u(x)w_i= e_K+\sum_{J\neq K, \# J=k} c_{J} e_{J}.
\end{align*}
Recall that $\omega_1\geq \cdots \geq \omega_m>0$ and $-\omega_{m+1}\leq \cdots \leq -\omega_d<0$. By definition, for any $i\in I(\ker \Tilde{x}\cap W)$ ($i\in J(\Tilde{x}W)$, resp.), $e_i$ is the slowest (largest, resp.) contracting (expanding, resp.) component of $u(x)w_i$ under the action of $a^{\omega}_1$. Hence, if $J\neq K$, $\#J=k$ and $c_J\neq 0$, then $\mu_J\leq \mu_K$. Therefore, the eigenvalue of $e_K$ with respect to $a^{\omega}_1$, which is $\exp(\mu_K)=\exp(\phi(\Tilde{x}W)-\psi(\ker \Tilde{x}\cap W))$, is the largest one. This finishes the proof.
\end{proof}

Given $A\in M_{d,d}$, we may write $A=(A_{ij})_{1\leq i,j\leq d}=(A_1,\cdots,A_d)$, where $A_i$ is a $d\times 1$ column vector for each $i$. Denote by
\[A_J=A_{j_1}\wedge\cdots\wedge A_{j_k}\in \bigwedge^k \mathbb{R}^d,\]
where $J=\{j_1<\cdots<j_k\}\subset \{1,\cdots,d-1\}$. Let $W=\langle w_1,\cdots,w_k \rangle$ be a $k$-dimensional linear subspace of $\mathbb{R}^d$ and $v_W=w_1\wedge\cdots\wedge w_k$. For convenience, we write 
\[w_i=\begin{pmatrix}
    w_{1i}\\
    \vdots\\
    w_{di}
\end{pmatrix}\in \mathbb{R}^d\]
for each $1\leq i \leq k$. For any $I,J\subset \{1,\cdots,d\}$ with $\# I=\# J=k$, let
\begin{align*}
    A(I,J)=(A_{ij})_{i\in I, j\in J}\in M_{k,k};\\
    v_W(J)=(w_{ji})_{j\in J,1\leq i\leq k}\in M_{k,k}.
\end{align*}
We observe the following simple lemma.

\begin{lemma}\label{Lemma: coefficient of e_I}
    In $\bigwedge^k \mathbb{R}^d$, the coefficient of $e_I$ in $Av_W$ is 
    \[\sum_{\#J=k} \det A(I,J) \det v_W(J),\]
where the summation is taken over all subsets $J$ of $\{1,\cdots,d\}$ of cardinality $k$.
\end{lemma}

\begin{proof}
    For $1\leq i \leq k$, we have 
    \[A w_i=A (\sum_{j=1}^d w_{ji} e_j)=\sum_{j=1}^d w_{ji} A_j.\]
So
\begin{align*}
    A v_W&=Aw_1\wedge \cdots \wedge A w_k
    =(\sum_{j=1}^d w_{j1} A_j)\wedge \cdots \wedge (\sum_{j=1}^d w_{jk} A_j).
\end{align*}
We observe that for any $J\subset \{1,\cdots, d\}$ with $\# J=k$, the component of $A_J$ in $A v_W$ is $\det v_W(J) A_J$, and the component of $e_I$ in $A_J$ is $\det A(I,J) e_I$. Therefore, the coefficient of $e_I$ in $Av_W$ is $\sum_{\# J=k} \det A(I,J) \det v_W(J)$.

\end{proof}

A lexicographic type order in the family of all subsets of cardinality $k$ of $\{1,\cdots,d\}$ will play a role in obtaining Corollaries \ref{Corollary: connectedness in 1 by d} and \ref{Corollary: connectedness in the dual version}:

\begin{definition}\label{Definition: an order of subsets}
We define an order $\prec_k$ in the family of all subsets of cardinality $k$ of $\{1,\cdots, d\}$ as follows: for two subsets $J=\{j_1<\cdots<j_k\}$ and $J'=\{j'_1<\cdots<j'_k\}$, $J'\prec_k J$ if 
\begin{align*}
  \text{there exists } r\in \{1,\cdots, k\} \text{ such that } j'_{r}<j_{r} \text{ and } j'_s= j_s \text{ for all } s\geq r+1, 
\end{align*}
where we set $j'_{k+1}=j_{k+1}=d+1$. Hence, $r=k$ means $j'_k<j_k$. It is straightforward to verify that $\prec_k$ defines a linear order. 

\end{definition}

\subsection{Successive minima}
For any nonzero vector $v\in \mathbb{R}^d$ and $t\geq 0$, we denote by $\lambda(a^{\omega}_t v)$ the smallest $\lambda> 0$ such that $a^{\omega}_t v\in \lambda \mathcal{C}$. Denote by 
\[L^{\omega}(v,t)=\log \lambda(a^{\omega}_t v).\]
Given a unimodular lattice $\Lambda$ and $t\geq 0$, there exist nonzero $v_1,\cdots,v_d\in \Lambda$ (depending on $t$) such that 
\[L^{\omega}_i(\Lambda,t)=L^{\omega}(v_i,t), \quad \forall 1\leq i\leq d.\]
We say such $v_1,\cdots,v_d\in \Lambda$ realize the successive minima of $a^{\omega}_t \Lambda$. In this case it is necessary that $v_i$ is primitive in $\Lambda$ for each $1\leq i\leq d$.

Note that $\Lambda^k:=\bigwedge^k \Lambda$ can be viewed as a unimodular lattice in $\bigwedge^k \mathbb{R}^d$. Consider the closed convex symmetric compound body $\mathcal{C}^k\subset \bigwedge^k \mathbb{R}^d$ given by 
\begin{align*}
    \mathcal{C}^k=\{\sum_{I} c_I e_I : |c_I|\leq 1, \forall I \text{ with } \# I=k\}.
\end{align*}
Let $N=\dim \bigwedge^k \mathbb{R}^d$. As in the case of $k=1$, for any $1\leq k\leq d-1$ we may also consider the successive minima $\lambda_i(a^{\omega}_t \Lambda^k), 1\leq i\leq N$ with respect to $\mathcal{C}^k$, and the corresponding function $\boldsymbol{L}^{\omega}(\Lambda^k,t)$ from $\mathbb{R}_{\geq 0}$ to $\mathbb{R}^N$. Given $t\geq 0$, let $v_1,\cdots,v_d\in \Lambda$ realize the successive minima of $a^{\omega}_t \Lambda$. Then Mahler's compound body theorem \cite[Theorem 3]{Mahler_2019_On+compound+convex+bodies+_MR4605022} asserts that 
\begin{align}\label{align: Mahler's theorem}
    L^{\omega}_1(\Lambda^k,t)\asymp_{+,d} L^{\omega}(v_1\wedge\cdots\wedge v_k,t),
\end{align}
where for two real numbers $A$ and $B$, $A\asymp_{+,d} B$ means there exists a constant $C$ depending only on $d$ such that $|A-B|\leq C$.
For any $t\geq 0$, if $L^{\omega}_k(\Lambda,t)<L^{\omega}_{k+1}(\Lambda,t)$, we denote by $W_k(t)$ the $k$-dimensional linear space spanned by $k$ linearly independent lattice points in $\lambda_k(a^{\omega}_t \Lambda)\mathcal{C}\cap a^{\omega}_t \Lambda$. We need the following observation: 

\begin{lemma}\label{Lemma: consequence of disconnectedness}\cite[Lemma 2.1]{Schmidt_Summerer_2009_Parametric_geometry_of_numbers_and_applications_MR2557854}
Let $k\in \{1,\cdots,d-1\}$. Suppose there exists $t_0\geq 0$ such that $L^{\omega}_k(\Lambda,t)<L^{\omega}_{k+1}(\Lambda,t)$ for all $t\geq t_0$. Let $W$ be the $k$-dimensional space spanned by $k$ linearly independent vectors $w_1,\cdots,w_k$ of $\Lambda$ realizing $L^{\omega}_1(\Lambda,t_0),\cdots,L^{\omega}_k(\Lambda,t_0)$. Then $W_k(t)=a^{\omega}_t W$ for all $t\geq t_0$.
\end{lemma}

\subsection{Pencils}
In this subsection, we will collect some basic properties of pencils.

\begin{proposition}\label{Proposition: pencils are Zariski closed}
Given an $(m,n)$-weight vector $\omega$, $k\in \{1,\cdots,d-1\}$, $a,b\in \mathbb{R}$. Let $W$ be a $k$-dimensional linear subspace of $\mathbb{R}^d$. Then the pencil $P^{\omega}_{W,a,b}$ defined as in Definition \ref{Definition: pencils} is a Zariski closed algebraic variety of $M_{m,n}$. 
\end{proposition}

\begin{remark}\label{Remark: finitely many pencils}
By the definition of $\psi,\phi$ ( see (\ref{align: definition of psi and phi}) ), $\psi$ and $\phi$ only take finitely many values. Therefore, for any $k$-dimensional linear subspace $W$, the family $\mathcal{P}^{\omega}_W$ of pencils defined by
    \[\mathcal{P}^{\omega}_W=\{\mathcal{P}^{\omega}_{W,a,b}:a,b\in \mathbb{R}\}\]
    has only finitely many elements.
\end{remark}

Proposition \ref{Proposition: pencils are Zariski closed} is well-known to experts. For completeness, we include a proof. Since $P^{\omega}_{W,a,b}=\{x\in  M_{m,n}:\psi(\ker\tilde{x}\cap W)\geq a\}\cap \{x\in M_{m,n}:\phi(\Tilde{x}W)\leq b\}$, it suffices to prove both these sets are Zariski closed. Firstly, let us prove that $\{x\in  M_{m,n}:\psi(\ker\tilde{x}\cap W)\geq a\}$ is Zariski closed.

\begin{lemma}
Let $\omega,k,a,W$ be as in Proposition \ref{Proposition: pencils are Zariski closed}. Then $\{x\in  M_{m,n}:\psi(\ker\tilde{x}\cap W)\geq a\}$ is Zariski closed.
\end{lemma}

\begin{proof}
For simplicity, we denote $E_i=\langle e_1,\cdots,e_i \rangle$ for $m+1\leq i\leq d$. We first make the following observation: since $\tilde{x}\in M_{m,d}$ is the first $m$ rows of $u(x)$, by the definition of $u(x)$ we have $\ker \Tilde{x}\cap E_m=\{0\}$. Moreover, we have $\dim(\ker \Tilde{x}\cap W)+\dim(\Tilde{x}W)=\# I(\ker \Tilde{x}\cap W)+\# J(\Tilde{x}W)=k$. For each $s\in \{1,\cdots,n\}$, let 
\[\mathcal{I}_s(a)=\{(i_1,\cdots,i_s):\omega_{m+i_1}+\cdots+\omega_{m+i_s}\geq a,  i_1<\cdots<i_s\in \{1,\cdots,n\}\}.\]
Then the cardinality of $\mathcal{I}_s(a)$ is finite. We have the following claim:
\begin{claim}\label{Claim: psi}
    \begin{align*}
 &\{x\in M_{m,n}:\psi(\ker \Tilde{x}\cap W)\geq a\}\\
&=\bigcup_{s=1}^n\bigcup_{(i_1,\cdots,i_s)\in \mathcal{I}_s(a)}\bigcap_{j=1}^s \{x\in M_{m,n}:\dim(\ker \tilde{x}\cap W\cap E_{m+i_j})\geq j\}.  
    \end{align*}
\end{claim}
Assume Claim \ref{Claim: psi} holds. Then it suffices to prove that for any $1\leq j\leq i$, the set
\[\{x\in M_{m,n}:\dim(\ker \tilde{x}\cap W\cap E_{m+i})\geq j\}\]
is Zariski closed. Indeed, we may write 
\[W\cap E_{m+i}=\langle w_1,\cdots,w_q \rangle,\]
for $q=\dim (W\cap E_{m+i})$ and $w_1,\cdots,w_q\in W$ linearly independent. We further denote
\[W_q=\left(w_1,\cdots,w_q\right)\in M_{d,q}\]
to be a $d\times q$ matrix whose $j$-th column is $w_j$. Note that for $x\in M_{m,n}$, we have $\dim(\ker \Tilde{x}\cap W \cap E_{m+i})\geq j$ if and only if $\dim(\Tilde{x} (W\cap E_{m+i}))\leq q-j$, which is further equivalent to $\rm{rank}(\Tilde{x} W_q)\leq q-j$. Observe that each entry of $\Tilde{x}W_q\in M_{m,q}$ is a linear equation of entries of $x$, and $\rm{rank}(\Tilde{x}W_q)\leq q-j$ is an algebraic condition. Therefore, $\{x\in M_{m,n}:\dim(\ker \tilde{x}\cap W\cap E_{m+i})\geq j\}$ is a Zariski closed algebraic variety.

Now it remains to prove Claim \ref{Claim: psi}. Note that if for all $s\in \{1,\cdots,n\}$, $\mathcal{I}_s(a)=\emptyset$, then both sets in Claim \ref{Claim: psi} are empty and there is nothing to prove. Therefore, we may assume there exists $s\in \{1,\cdots,n\}$ such that $\mathcal{I}_s(a)\neq \emptyset$. If $x\in M_{m,n}$ is such that $\psi(\ker \Tilde{x}\cap W)\geq a$, then $\sum_{i\in I(\ker \Tilde{x}\cap W)} \omega_i\geq a$. Let $s=\dim(\ker\Tilde{x}\cap W)$ and $I(\ker\tilde{x}\cap W)=\{m+i_1,\cdots,m+i_s\}$. Then $(i_1,\cdots,i_s)\in \mathcal{I}_s(a)$. By the definition of $I(\ker\tilde{x}\cap W)$, for any $1\leq j\leq s$ we have 
\[\dim(\ker\Tilde{x}\cap W\cap E_{m+i_j})>\dim(\ker\Tilde{x}\cap W\cap E_{m+i_j-1}).\]
As $\dim(\ker \Tilde{x}\cap W\cap E_m)=0$, we have $\dim(\ker\Tilde{x}\cap W\cap E_{m+i_j})=j$ and this proves one side inclusion of the sets.

Let $x\in M_{m,n}$ be such that there exist $s\in \{1,\cdots,n\}$, $(i_1,\cdots,i_s)\in \mathcal{I}_s(a)$ such that for all $1\leq j\leq s$, $\dim(\ker\Tilde{x}\cap W\cap E_{m+i_j})\geq j$. Then by induction we see that for any $1\leq j \leq s$, there exists $1\leq r_j\leq i_j$ such that 
\[\dim(\ker \Tilde{x}\cap W\cap E_{m+r_j-1})<\dim(\ker \Tilde{x}\cap W\cap E_{m+r_j}),\]
and $r_1<\cdots<r_s$. Thus $\{m+r_1,\cdots,m+r_s\}\subset I(\ker\Tilde{x}\cap W)$. As $\omega_{m+1}\geq \cdots\geq\omega_{d}>0$, we have $\sum_{j=1}^s \omega_{m+r_j}\geq \sum_{j=1}^{s}\omega_{m+i_j}\geq a$. This proves the claim, and finishes the proof of the lemma.
\end{proof}

\begin{lemma}
Let $\omega,k,b,W$ be as in Proposition \ref{Proposition: pencils are Zariski closed}. Then $\{x\in  M_{m,n}:\phi(\tilde{x} W)\leq b\}$ is Zariski closed.
\end{lemma}

\begin{proof}
For simplicity, we denote $E'_j=\langle e_j,\cdots, e_m \rangle$ for $1\leq j\leq m$ and $E'_{m+1}=0$. For each $1\leq s \leq m$, we let
\[\mathcal{J}_s(b)=\{(j_1,\cdots,j_s):\sum_{i=1}^s \omega_{j_i}> b, \quad j_1<\cdots<j_s\in \{1,\cdots,m\}\}.\]
The cardinality of $\mathcal{J}_s(b)$ is finite. First we observe that if $b<0$, then $\{x\in  M_{m,n}:\phi(\tilde{x} W)\leq b\}$ is empty and there is nothing to prove. Now suppose $b>0$, we have the following claim:
\begin{claim}\label{Claim: phi}
\begin{align*}
    &\{x\in M_{m,n}:\phi(\Tilde{x}W)> b\}=\{x\in M_{m,n}:\dim(\Tilde{x}W)>0\}\cap\\
&\bigcup_{s=1}^m\bigcup_{(j_1,\cdots,j_s)\in \mathcal{J}_s(b)}\bigcap_{i=1}^s \{x\in M_{m,n}:\dim(\Tilde{x}W\cap E'_{j_i+1})< s-i+1\}.
\end{align*}
\end{claim}

Assume Claim \ref{Claim: phi} holds. Then it suffices to prove that $\{x\in M_{m,n}:\dim(\Tilde{x}W)=0\}$ and $\{x\in M_{m,n}:\dim(\Tilde{x}W\cap E'_{j})\geq i\}$ for $1\leq i,j\leq m$ are all Zariski closed. For this, suppose $W$ is spanned by vectors $w_1,\cdots,w_k$ and we let $W_k=(w_1,\cdots,w_k)\in M_{d,k}$ be a $d\times k$ matrix whose $i$-th column is $w_i$. Then $\Tilde{x}W_k\in M_{m,k}$. Note that $x\in M_{m,n}$ satisfies $\dim(\Tilde{x}W)=0$ if and only if $\rm{rank}(\tilde{x}W_k)=0$. As each entry of $\Tilde{x}W_k$ is a linear equation of entries of $\Tilde{x}$, $\rm{rank}(\tilde{x}W_k)=0$ is an algebraic condition on $x$. Hence $\{x\in M_{m,n}:\dim(\Tilde{x}W)=0\}$ is Zariski closed. 

We observe that for $x\in M_{m,n}$, $\dim(\tilde{x}W\cap E'_j)\geq i$ if and only if the first $j-1$ rows of the matrix $\Tilde{x}W_k$ has rank $\leq k-i$. This is an algebraic condition on $x$. Hence, $\{x\in M_{m,n}:\dim(\Tilde{x}W\cap E'_{j})\geq i\}$ is Zariski closed.

Now it remains to prove Claim \ref{Claim: phi}. Suppose $x\in M_{m,n}$ satisfies $\phi(\Tilde{x}W)> b>0$. We have $\dim(\Tilde{x}W)\geq 1$ and $J(\Tilde{x}W)=\{j_1,\cdots,j_s\}$ for some $j_1<\cdots<j_s\in \{1,\cdots,m\}$ with $\sum_{i=1}^s \omega_{j_i}> b$. By the definition of $J(\Tilde{x}W)$, we have
\[\dim(\Tilde{x}W\cap E'_{j_i})=s-i+1, \quad 1\leq i\leq s.\]
Therefore,
\[\dim(\Tilde{x}W\cap E'_{j_i+1})<s-i+1, \quad 1\leq i\leq s.\]
This finishes the proof of one side inclusion of the claim. 

Suppose that $x\in M_{m,n}$ satisfies that $\dim(\Tilde{x}W)\geq 1$ and there exist $s\in \{1,\cdots,m\}$, $(j_1,\cdots,j_s)\in \mathcal{J}_s(b)$ such that $\dim(\Tilde{x}W\cap E'_{j_i+1})< s-i+1$ for all $1\leq i\leq s$. By induction, we obtain $1\leq r_1<\cdots<r_s\leq m$ such that $r_i\leq j_i$ for $1\leq i \leq s$ and
\[\dim(\Tilde{x}W\cap E'_{r_i+1})<\dim(\Tilde{x}W\cap E'_{r_i}), \quad 1\leq i\leq s.\]
Therefore, $\{r_1,\cdots,r_s\}\subset J(\Tilde{x}W)$. As $\omega_1\geq \cdots\geq \omega_m>0$ and $\sum_{i=1}^s \omega_{j_i}> b$, we obtain $\phi(\Tilde{x}W)\geq\sum_{i=1}^s \omega_{r_i}> b$. The lemma is proven.
\end{proof}

\section{Proofs of main theorems and their corollaries }\label{Section: proof of main theorem and corollary}

\begin{proof}[Proof of Theorem \ref{Theorem: algebraic connectedness}]
 Suppose that the assumption of Theorem \ref{Theorem: algebraic connectedness} is satisfied for the lattice $\Lambda$ and some $k\in \{1,\cdots,d-1\}$. We proceed by contradiction. Assume that $\boldsymbol{L}^{\omega}(\Lambda,\cdot)$ is disconnected at $k$, then there exists $t_0\geq 0$ such that for all $t\geq t_0$, $L^{\omega}_k(\Lambda,t)<L^{\omega}_{k+1}(\Lambda,t)$. By Lemma \ref{Lemma: consequence of disconnectedness}, there is a $k$-dimensional space $W$ spanned by vectors of $\Lambda$ such that $W_k(t)=a^{\omega}_t W$ for all $t\geq t_0$. Now for any $t\geq t_0$, {let $\{w_i(t)\in a_t^{\omega}\Lambda\setminus \{0\}:1\leq i\leq k\}$ be the set of vectors such that they span the same subspace as the vectors in $a^{\omega}_{-t}(\lambda_k(a^{\omega}_t \Lambda)\mathcal{C}\cap a^{\omega}_{t}\Lambda \setminus \{0\})$.}Therefore, {$w_1(t)\wedge \cdots \wedge w_k(t)=v_W$, and} by (\ref{align: Mahler's theorem}) we have
 \begin{align*}
      L^{\omega}_1(\Lambda^k,t)\asymp_{+,d} L^{\omega}(v_W, t).
 \end{align*}
By assumption, the largest eigenvalue in the decomposition of $v_W$ into eigenvectors of $a^{\omega}_1$ is $\exp(\mu)$ for some $\mu>0$. Hence, by the definition of $ L^{\omega}(v_W, t)$,
\begin{align}\label{align: L1 is too large}
    L^{\omega}_1(\Lambda^k,t)\asymp_{+,d} L^{\omega}(v_W, t)>t \mu/2
\end{align}
for all large enough $t\geq t_0$. {However, since $\Lambda^k$ is unimoduar and $a^{\omega}_t$ is of determinant $1$, one always has $L^{\omega}_1(\Lambda^k,t)\leq 0$.} 
Therefore, (\ref{align: L1 is too large}) yields the desired contradiction.
\end{proof}

\begin{proof}[Proof of Theorem \ref{Theorem: generic behavior of connectedness}]
By Proposition \ref{Proposition: pencils are Zariski closed}, any pencil is an algebraic variety of $M_{m,n}$. Let $W$ be a $k$-dimensional linear subspace defined over $\mathbb{Q}$. Since $\mathcal{M}$ is a connected analytic submanifold and not contained in any proper rational weakly constraining $k$-pencil, we conclude that for any $a,b$ satisfying $b\leq a$, $\mathcal{M}\cap\mathcal{P}^{\omega}_{W,a,b}$ is a proper algebraic subvariety of $\mathcal{M}$. Thus ${\mathcal{M}}\cap\mathcal{P}^{\omega}_{W,a,b}$ has Lebesgue measure zero in ${\mathcal{M}}$. Since there are only finitely many pencils in $\mathcal{P}^{\omega}_W=\{\mathcal{P}^{\omega}_{W,a,b}:a,b\in \mathbb{R}\}$ for fixed $W$ (see Remark \ref{Remark: finitely many pencils}), and there are only countably many $k$-dimensional linear subspaces defined over $\mathbb{Q}$, we conclude that the intersection of ${\mathcal{M}}$ with the union of all rational weakly constraining $k$-pencils has Lebesgue measure zero.

As a consequence, Lebesgue almost every point $x\in \mathcal{M}$ is not contained in any proper rational weakly constraining $k$-pencil. By Lemma \ref{Lemma: highest eigenvalue determined by pencil}, for any $k$-dimensional subspace $W$ defined over $\mathbb{Q}$, the largest eigenvalue of $u(x)v_W$ in the decomposition into eigenvectors of $a^{\omega}_1$ is $\exp(\phi(\Tilde{x}W)-\psi(\ker \Tilde{x}\cap W))$, which is strictly greater than $1$. Hence, Theorem \ref{Theorem: algebraic connectedness} implies $\boldsymbol{L}^{\omega}(\Lambda_x,\cdot)$ is connected at $k$.
\end{proof}

\begin{proof}[Proof of Corollary \ref{Corollary: a manifold not contained in affine plane is connected}]

When $m=1$, $M_{1,d}(\mathbb{R})\cong \mathbb{R}^d$. Hence, we can view any $\Tilde{x}\in M_{1,d}(\mathbb{R})$ as a vector in $\mathbb{R}^d$. Given a $k$-dimensional rational linear subspace $W$ of $\mathbb{R}^d$, if $x\in \mathcal{P}^{\omega}_{W,a,b}$ for some $b\leq a$, then necessarily $\psi(\ker \Tilde{x}\cap W)> 0$. So there exists a nonzero $w\in W$ such that $\Tilde{x}\boldsymbol{\cdot} w=0$, where $\boldsymbol{\cdot}$ denotes the standard inner product. Therefore, ${x}$ is contained in a proper affine hyperplane of $\mathbb{R}^d$. As a consequence, when $m=1$, any proper rational weakly constraining pencil is contained in a proper affine hyperplane of $\mathbb{R}^d$.

If ${\mathcal{M}}$ is not contained in any proper affine hyperplane of $\mathbb{R}^{d}$, then the assumption of Theorem \ref{Theorem: generic behavior of connectedness} is satisfied for any $k\in \{1,\cdots,d-1\}$. Therefore, the corollary follows.

\end{proof}

\begin{proof}[Proof of Corollary \ref{Corollary: connectedness in 1 by d}]
Let $A=u(x)$, that is,
\[A=\begin{pmatrix}
        1 & & & \xi_1\\
          & \ddots & & \vdots\\
        & &   1 & \xi_{d-1}\\
         & & & 1
    \end{pmatrix}.\]
Let $k\in \{1,\cdots, d-1\}$. In what follows $I,J$ are subsets of $\{1,\cdots, d\}$ of cardinality $k$. It is then straightforward to verify the following: 
\begin{itemize}
    \item[1)] If $d\in I$ and $d\notin J$, then $\det A(I,J)=0$.

    \item[2)] If $d\notin I$, then
\begin{align}\label{align: value of det A(I,J)}
    \det A(I,J)=\begin{cases}
        1 & \text{ if } I=J\\
        0 & \text{ if } d\notin J, I\neq J \text{ or } d\in J, \# I\setminus J \geq 2\\
        \pm \xi_{i_s} & \text{ if } d\in J, \{i_s\}=I\setminus J
    \end{cases}
\end{align}
The sign of $\xi_{i_s}$ depends on $I,J$. It will become apparent that the sign of $\xi_{i_s}$ does not play any role in our analysis.
\end{itemize}

Let $W$ be a $k$-dimensional linear subspace spanned by vectors of $\mathbb{Z}^d$ and $v_W=w_1\wedge\cdots \wedge w_k\in \bigwedge^k \mathbb{Z}^d$ represent $W$. We want to show that for any such $W$, there exists $I\subset \{1,\cdots, d-1\}$ of cardinality $k$ such that the $e_I$-component of $Av_W$ is nonzero. This would imply that $x$ is not contained in any proper weakly rational constraining $k$-pencil.

First we make reduction on the assumption that there exist $\{n_1,\cdots,n_k\}\subset \{1,\cdots,d-1\}$ such that $1,\xi_{n_1},\cdots,\xi_{n_k}$ are linearly independent over $\mathbb{Q}$. We claim that without loss of generality, we may assume $1,\xi_1,\cdots,\xi_k$ are linearly independent over $\mathbb{Q}$. Indeed, given $r,s\in \{1,\cdots,d-1\}$ with $r<s$, we denote by $\sigma_{rs}$ the permutation matrix in $SL_d(\mathbb{R})$, i.e., the $r$-th column of $\sigma_{rs}$ is $\pm e_s$ and $s$-th column is $\pm e_r$, while $i$-th column of $\sigma_{rs}$ is $e_i$ for $i\neq r,s$. Here we include $\pm$ sign to ensure that $\sigma_{rs}\in \rm{SL}_d(\mathbb{Z})$. Thus, up to sign, acting on the left (right, resp.) on a $d\times d$ matrix, $\sigma_{rs}$ permutes $r$-th and $s$-th rows (columns, resp.) of this matrix. Moreover, for any $g\in SL_d(\mathbb{R})$, $g\sigma_{rs}\mathbb{Z}^d=g\mathbb{Z}^d$ and $\sigma_{rs} \mathcal{C}=\mathcal{C}$. Therefore, for any $t\geq 0$ and any $i\in \{1,\cdots,d\}$, we have
\[\lambda_i(a^{\omega}_t u(x) \mathbb{Z}^d)=\lambda_i(\sigma_{rs}^{-1}a^{\omega}_t\sigma_{rs}\sigma^{-1}_{rs}u(x)\sigma_{rs} \mathbb{Z}^d).\]

As a consequence, conjugating by {a} permutation matrix does not change the successive minima, thus {keeps} the connectedness property. Depending on $n_1,\cdots,n_k$, we may choose finitely many permutation matrices $\sigma_{rs}$ for $r,s\in \{1,\cdots,d-1\}$ and set $\sigma$ to be the product of them such that the first $k$ rows of $\sigma^{-1}u(x)\sigma$ is
\[\begin{pmatrix}
       1 && & & \xi_{n_1}\\
          &\ddots & & & \vdots\\
        & &1 & \cdots   & \xi_{n_k}
\end{pmatrix}.\]
We may write $\sigma^{-1} a^{\omega}_t \sigma=a^{\sigma(\omega)}_t$ for $\sigma(\omega)=(\omega_{\sigma(1)},\cdots,\omega_{\sigma(d-1)},1)$, where we denote $\sigma(i), 1\leq i\leq d-1$ to be the corresponding permutation of the set $\{1,\cdots,d-1\}$ under $\sigma$. Although $\sigma(\omega)$ may no longer be a weight vector by definition, assume we have shown that for any $k$-dimensional rational linear subspace $W$, there exists $I(W)\subset \{1,\cdots,d-1\}$ such that the $e_{I(W)}$-component of $\sigma^{-1}u(x)\sigma v_W$ is nonzero. Then for any $k$-dimensional rational linear subspace $W$, as $a^{\omega}_t u(x)v_W=\sigma \sigma^{-1}a^{\omega}_t u(x)\sigma v_{\sigma^{-1}W}$, the $e_{\sigma I(\sigma^{-1}W)}$-component of $u(x)v_W$ is nonzero, so $x$ is not contained in any proper weakly rational constraining $k$-pencil. Hence, we can assume that $1,\xi_1,\cdots,\xi_k$ are linearly independent over $\mathbb{Q}$.

To continue, we note that $\det v_W(J)\in \mathbb{Z}$ for any $J$, where $J\subset\{1,\cdots, d\}$ is of cardinality $k$. Moreover, there exists at least one $J$ of cardinality $k$ such that $\det v_W(J)\neq 0$. By Lemma \ref{Lemma: coefficient of e_I}, we divide the analysis into two cases depending on the values of $\det v_W(J)$.

Case I: For all $J$ with $d\in J$, $\det v_W(J)=0$. Then by Lemma \ref{Lemma: coefficient of e_I}, we may write
    \begin{align*}
        A v_W=\sum_{I, d\notin I} (\sum_{J,d\notin J}\det A(I,J)\det v_W(J))e_I
    \end{align*}
As $Av_W\neq 0$, there exists some $I_0\subset \{1,\cdots,d-1\}$ such that the $e_{I_0}$-component is nonzero.

Case II: There exists some $J$ with $d\in J$ such that $\det v_W(J)\neq 0$.

Consider the following family of subsets:
\begin{align*}
    \mathcal{J}=\{J\subset \{1,\cdots,d\}: \# J=k, d\in J, \text{ and } \det v_W(J)\neq 0\}.
\end{align*}
By the assumption of Case II, $\mathcal{J}$ is nonempty.
Let $J_0$ be the smallest element in $\mathcal{J}$ with respect to the order $\prec_k$ (see Definition \ref{Definition: an order of subsets}). We write 
\[J_0=\{j_1<\cdots<j_{k-1}<d\}.\]

Now we choose $I_0 \subset \{1,\cdots,d-1\}$ of cardinality $k$ such that the $e_{I_0}$-component in $Av_W$ is nonzero as follows.
Let $i_0$ be the smallest element in $ \{1,\cdots,d-1\}\setminus \{j_1,\cdots,j_{k-1}\}$, then necessarily $1\leq i_0\leq k$. Note that one and only one of the following three situations holds: (i) $i_0<j_1$, in which case we must have $i_0=1$; (ii) $i_0>j_{k-1}$, in which case we have $i_0=k$ and $j_i=i$ for $1\leq i\leq k-1$; (iii) there exists $s\in \{1,\cdots,k-2\}$ such that $j_s<i_0<j_{s+1}$, in which case we have $j_s<k$.

In any of the above situations, we set $I_0=(J_0\setminus \{d\})\cup \{i_0\}$, i.e.
\[I_0=\{j_1<\cdots<j_s<i_0<j_{s+1}<\cdots<j_{k-1}\}.\]
We claim that $e_{I_0}$-component is nonzero. To prove the claim, we investigate those $J$ such that $\det A(I_0,J)\neq 0$. That is, by (\ref{align: value of det A(I,J)}), those $J$ satisfying $J=I_0$, or $d\in J$ and $\# I_0\setminus J=1$. Depending on the choice of $J$, we have the following:
\begin{itemize}
    \item [(1)] $J=I_0$. Then $\det A(I_0,I_0)=1$. We set $c_0=\det v_W(I_0)$.

    \item[(2)] $d\in J$ and $I_0\setminus J=\{j_r\}$ for some $j_r<i_0$. Then $\det A(I_0,J)=\pm \xi_{j_r}$. Note that by the choice of $i_0$, we have $j_r<k$. We set $c_{j_r}=\rm{sgn}(\det A(I_0,J))\det v_W(J)$.

    \item[(3)] $d\in J$ and $I_0\setminus J=\{j_r\}$ for some $j_r>i_0$. Then $\det A(I_0,J)=\pm \xi_{j_r}$. Since $J\prec_k J_0$, the choice of $J_0$ implies $\det v_W(J)=0$.
    
    \item[(4)] $J=J_0$, then $\det A(I_0,J_0)=\pm \xi_{i_0}$. By assumption we have $\det v_W(J_0)\neq 0$. We set $c_{i_0}=\rm{sgn}(\det A(I_0,J_0))\det v_W(J_0)$.

\end{itemize}
Hence by Lemma \ref{Lemma: coefficient of e_I}, the coefficient of $e_{I_0}$ is 
\[\sum_J \det A(I_0,J)\det v_W(J)= c_0+\sum_{i=1}^k c_i \xi_i,\]
where $c_0,\cdots,c_k\in \mathbb{Z}$. Note that $i_0\leq k$ and $c_{i_0}=\det v_W(J_0)\neq 0$. By linear independency of $1,\xi_1,\cdots,\xi_k$ over $\mathbb{Q}$, $e_{I_0}$-component is nonzero and he corollary follows.

\end{proof}

\begin{proof}[Proof of Corollary \ref{Corollary: connectedness in the dual version}]
The proof is similar to that of Corollary \ref{Corollary: connectedness in 1 by d}.

Let $A=u(y)$, that is,
\begin{align*}
   A=\begin{pmatrix}
        1& \xi_1 &\cdots& \xi_{d-1}\\
         &  1 &  &\\
         & & \ddots &\\
         &         &  & 1
    \end{pmatrix}.
\end{align*} 
Given $k\in \{1,\cdots,d-1\}$, it is straightforward to verify the following, where $I,J$ below are subsets of $\{1,\cdots,d\}$ of cardinality $d-k$:
\begin{itemize}
    \item[1)] If $1\notin I$, then 
\begin{align*}
    \det A(I,J)=\begin{cases}
        1 & \text{ if } J=I\\
        0 & \text{ if } J\neq I
    \end{cases}.
\end{align*}

\item[2)] If $1\in I$, then
\begin{align*}
    \det A(I,J)=\begin{cases}
        1 & \text{ if } J=I\\
        0 & \text{ if } 1\in J, J\neq I \text{ or }1\notin J, \# J\setminus I \geq 2\\
        \pm \xi_{j_s -1}& \text{ if } 1\notin J, \{j_s\}=J\setminus I
    \end{cases}.
\end{align*}
\end{itemize}

By the similar reduction made in the proof of Corollary \ref{Corollary: connectedness in 1 by d}, we may assume that $1,\xi_1,\cdots,\xi_k$ are linearly independent over $\mathbb{Q}$. Let $W$ be a $(d-k)$-dimensional subspace spanned by vectors of $\mathbb{Z}^d$. Then there exist $d-k$ linearly independent vectors $w_1,\cdots, w_{d-k}\in \mathbb{Z}^d$ such that $W$ is spanned by $ w_1,\cdots,  w_{d-k}$. Set $v_W=w_1\wedge\cdots \wedge w_{d-k}\in \bigwedge^{d-k} \mathbb{R}^d$. We want to show that there exists $I\subset \{1,\cdots, d\}$ of cardinality $d-k$ such that $1\in I$ and the $e_I$-component of $Av_W$ is nonzero.

First we note that $\det v_W(J)\in \mathbb{Z}$ for any $J$, where $J\subset\{1,\cdots, d\}$ is of cardinality $d-k$. Moreover, there exists at least one $J$ of cardinality $d-k$ such that $\det v_W(J)\neq 0$. By Lemma \ref{Lemma: coefficient of e_I}, we divide the analysis into two cases depending on the values of $\det v_W(J)$.

Case I: For all $J$ such that $1\notin J$, $\det v_W(J)=0$. Then for any $J$ with $1\in J$ we have 
\begin{align*}
    \det A(I,J)=\begin{cases}
        0 & \text{ if } I\neq J \\
        1 & \text{ if } I=J
    \end{cases}.
\end{align*}
By Lemma \ref{Lemma: coefficient of e_I}, we may write
\begin{align*}
    A v_W= \sum_{I=J,1\in J} \det A(I,J)\det v_W(J) e_I.
\end{align*}
As $Av_W\neq 0$, there exists $I=J$ with $1\in J$ such that the $e_I$-component is nonzero.

Case II: There exists $J$ such that $1\notin J$ and $\det v_W(J)\neq 0$. Consider
\[\mathcal{J}=\{J\subset \{1,\cdots,d\}: \# J=d-k, 1\notin J, \det v_W(J)\neq 0\}.\]
By the assumption of Case II, $\mathcal{J}$ is nonempty. Let $J_0\in \mathcal{J}$ be the largest element with respect to the order $\prec_{d-k}$. We may write 
\[J_0=\{j_1<\cdots<j_{d-k}\}.\]
Then since $1\notin J_0$, we have 
\[2\leq j_1\leq \cdots<j_{d-k}\leq d.\]
Therefore, $j_1\leq k+1$.
Let $I_0=(J_0\setminus\{j_1\})\cup \{1\}$, i.e.
\[I_0=\{1<j_2<\cdots<j_{d-k}\}.\]
We claim that the $e_{I_0}$-component of $Av_W$ is nonzero. To prove this claim, we investigate those $J$ such that $\det A(I_0,J)\neq 0$. That is, those $J$ such that $J=I_0$, or $1\notin J$ and $\# J\setminus I_0=1$. Depending on the choice of $J$, we have the following:
\begin{itemize}
    \item[(1)] $J=I_0$. Then $\det A(I_0,J)=1$.

    \item[(2)] $J=(I_0\setminus \{1\})\cup \{j_0\}$ for some $j_0\leq j_1$. Then $\det A(I_0,J)=\pm \xi_{j_0-1}$. In particular $j_0\leq k+1$.

    \item[(3)]  $J=(I_0\setminus \{1\})\cup \{j_0\}$ for some $j_0 >j_1$. Then $J_0\prec_{d-k} J$. By the choice of $J_0$, this implies $\det v_W(J)=0$.
\end{itemize}
Therefore, the coefficient of $e_{I_0}$ is 
\[\sum_J \det A(I_0,J)\det v_W(J)= c_0+\sum_{i=1}^k c_i \xi_i,\]
where $c_i\in \mathbb{Z}$ and $c_{j_1 -1}\neq 0$. By linear independency of $1,\xi_1,\cdots,\xi_k$ over $\mathbb{Q}$, the $e_{I_0}$-component is nonzero. This completes the proof.

\end{proof}

\begin{proof}[Proof of Proposition \ref{Proposition: matrix connectedness}]
(a) Let $k\in \{m,\cdots,d-1\}$ and $W\subset \mathbb{R}^d$ be a $k$-dimensional rational linear subspace. We want to show that there exists $I\subset \{1,\cdots,d\}$ of cardinality $k$ such that $\{1,\cdots,m\}\subset I$ and the coefficient of $e_I$ in $u(x)v_W\in \bigwedge^k \mathbb{R}^d$ is nonzero. First let $A=u(x)$. Then it is straightforward to verify that for any $I,J\subset \{1,\cdots,d\}$ of cardinality $k$ and $\{1,\cdots,m\}\subset I$, we have 
\begin{align*}
\det A(I,J)=\begin{cases}
    1 & \text{ if } I=J\\
  \pm \det(x_{i,j})_{i\in I\setminus J,j+m\in J\setminus I} & \text{ if } \emptyset \neq I\setminus J\subset \{1,\cdots,m\}\\
  0 & \text{ if } I\setminus J\not\subset \{1,\cdots,m\}
\end{cases}
\end{align*}
When $\emptyset \neq I\setminus J\subset \{1,\cdots,m\}$, the sign of $\det(x_{i,j})_{i\in I\setminus J,j+m\in J\setminus I}$ depends on $I$ and $J$. We first observe that if $I\setminus J\subset \{1,\cdots,m\}$, then $\# I\setminus J=\# J\setminus I\leq \min\{m,d-k\}$. This is because $I\setminus J\subset \{1,\cdots,m\}$ would imply $\#I\setminus J=\# J\setminus I\leq m$. Moreover, $\# I\setminus J\leq d-k$. For otherwise, we have $\# I\cup J>d$, which is a contradiction. We may choose an integral basis of $W$ such that for any $J\subset \{1,\cdots,d\}$ of cardinality $k$, $\det v_W(J)\in \mathbb{Z}$. We divide the analysis into two cases depending on $\det v_W(J)$.

Case I: For all $J\subset \{1,\cdots,d\}$ of cardinality $k$ such that $\{1,\cdots,m\}\not\subset J$, $\det v_W(J)=0$. As $\dim W=k$, there exists $J_0\subset \{1,\cdots,d\}$ of cardinality $k$ such that $\{1,\cdots,m\}\subset J_0$ and $\det v_W(J_0)\neq 0$. Let $I_0=J_0$, then $\det A(I_0,J_0)=1$. The coefficient of $e_{I_0}$ in $u(x)v_W$ is 
\[\det v_W(J_0)+\sum_{J:I\setminus J\subset \{1,\cdots,m\}} \det A(I_0,J)\det v_W(J).\]
Let $\mathcal{R}=\{j-m:j\in \{1,\cdots,d\}\setminus I_0\}\subset \{1,\cdots,n\}$. Then $\# \mathcal{R}\leq d-k$. By the assumption on $\mathcal{D}(\mathcal{R})$, the coefficient of $e_{I_0}$ is nonzero, and we are done.

Case II: There exists $J_0\subset \{1,\cdots,d\}$ of cardinality $k$ such that $\{1,\cdots,m\}\not\subset J_0$ and $\det v_W(J_0)\neq 0$. We choose arbitrary $k-m$ elements $j_1,\cdots,j_{k-m}\in J_0\setminus \{1,\cdots,m\}$ and let $I_0=\{1,\cdots,m\}\cup\{j_1,\cdots,j_{k-m}\}$. Then $\# I_0=k$, $I_0\neq J_0$ and $I_0\setminus J_0\subset \{1,\cdots,m\}$. The coefficient of $e_{I_0}$ in $u(x)v_W$ is
\[\det v_W(I_0,J_0)\det A(I_0,J_0)+\det v_W(I_0)+\sum \det A(I_0,J)\det v_W(J),\]
where the last summation on the far right runs through all $J$ of cardinality $k$ such that $J\neq J_0$ and $I\setminus J\subset \{1,\cdots,m\}$.
Let $\mathcal{R}=\{j-m:j\in \{1,\cdots,d\}\setminus I_0\}\subset \{1,\cdots,n\}$. Then $\# \mathcal{R}\leq d-k$. By the assumption on $\mathcal{D}(\mathcal{R})$, the coefficient of $e_{I_0}$ is nonzero. This finishes the proof of (a) in the proposition.

(b) Let $k\in \{1,\cdots,m\}$ and $W$ be a $k$-dimensional rational linear subspace of $\mathbb{R}^d$. Let $A=u(x)$. For $I\subset \{1,\cdots,m\},J\subset \{1,\cdots,d\}$ such that $\#I=\#J=k$, it is straightforward to verify that 
\begin{align*}
    \det A(I,J)=\begin{cases}
        1 & \text{ if } I=J\\
    \pm \det(x_{i,j})_{i\in I\setminus J, j+m\in J\setminus I}& \text{ if } J\setminus I\subset \{m+1,\cdots,d\}\\
    0 & \text{ if } J\setminus I\not\subset \{m+1,\cdots,d\}
    \end{cases}
\end{align*}
We observe that $\# I\setminus J=\# J\setminus I\leq \min\{m,d-k\}$. This is because $I\subset \{1,\cdots,m\}$ and $\# I\setminus J> d-k$ would yield $\# I\cup J>d$. As in (a), we divide our analysis into two cases.

Case I: For all $J\subset \{1,\cdots,d\}$ of cardinality $k$ and $J\subset \{1,\cdots,m\}$, $\det v_W(J)=0$. Then there exists $J_0\subset \{1,\cdots,m\}$ of cardinality $k$ such that $\det v_W(J_0)\neq 0$. Let $I_0=J_0$, then the coefficient of $e_{I_0}$ in $u(x)v_W$ is 
\[\det v_W(J_0)+\sum \det A(I_0,J)\det v_W(J),\]
where the summation runs through all $J\subset \{1,\cdots,d\}$ of cardinality $k$ such that $J\neq J_0$ and $J\setminus I_0\subset \{m+1,\cdots,d\}$. By the assumption on $\mathcal{D}(\mathcal{R})$, the coefficient of $e_{I_0}$ is nonzero.

Case II: There exists $J_0\subset \{1,\cdots,d\}$ of cardinality $k$ such that $J_0\not\subset \{1,\cdots,m\}$ and $\det v_W(J_0)\neq 0$. Choose any $I_0\subset \{1,\cdots,m\}$ of cardinality $k$ such that $J_0\cap \{1,\cdots,m\}\subset I_0$. Then $J_0\subset I_0\subset \{m+1,\cdots,d\}$. The coefficient of $e_{I_0}$ in $u(x)v_W$ is 
\[\det A(I_0,J_0)\det v_W(J_0)+\sum \det A(I_0,J)\det v_W(J),\]
where the summation runs through all $J\subset \{1,\cdots,d\}$ of cardinality $k$ such that $J\neq J_0$ and $J\setminus I_0\subset \{m+1,\cdots,d\}$. By the assumption on $\mathcal{D}(\mathcal{R})$, we conclude that the coefficient of $e_{I_0}$ is nonzero. This finishes the proof of (b).
\end{proof}

\noindent$\textbf{Acknowledgment.}$
We thank Nicolas de Saxc\'e for helpful discussions and the anonymous referee for suggestions improving the paper. Han Zhang acknowledges the support of startup grant (grant no. NH10701623) of Soochow University.

\bibliography{references}

\begin{thebibliography}{10}

\bibitem{Aka_Breuillard_Rosenzweig_Saxce_2018_Diophantine+approximation_MR3777412}
Menny Aka, Emmanuel Breuillard, Lior Rosenzweig, and Nicolas de~Saxc\'{e}.
\newblock Diophantine approximation on matrices and {L}ie groups.
\newblock {\em Geom. Funct. Anal.}, 28(1):1--57, 2018.

\bibitem{Bugeaud_Laurent_2010_On+transfer+inequalities+in+Diophantine+approximation+II_MR2609309}
Yann Bugeaud and Michel Laurent.
\newblock On transfer inequalities in {D}iophantine approximation. {II}.
\newblock {\em Math. Z.}, 265(2):249--262, 2010.

\bibitem{Dani_1985_Divergent_trajectories_of_flows_on_homogeneous_spaces_MR794799}
S.~G. Dani.
\newblock Divergent trajectories of flows on homogeneous spaces and {D}iophantine approximation.
\newblock {\em J. Reine Angew. Math.}, 359:55--89, 1985.

\bibitem{Das_Fishman_Simmons_Urbanski_2024_A+variational+principle_MR4671568}
Tushar Das, Lior Fishman, David Simmons, and Mariusz Urba\'{n}ski.
\newblock A variational principle in the parametric geometry of numbers.
\newblock {\em Adv. Math.}, 437:Paper No. 109435, 130, 2024.

\bibitem{German_2012_On+Diophantine+exponents+and+Khintchine's+transference+principle_MR2988525}
Oleg~N. German.
\newblock On {D}iophantine exponents and {K}hintchine's transference principle.
\newblock {\em Mosc. J. Comb. Number Theory}, 2(2):22--51, 2012.

\bibitem{Mahler_2019_On+compound+convex+bodies+_MR4605022}
Kurt Mahler.
\newblock On compound convex bodies. {I}.
\newblock {\em Doc. Math.}, (Extra vol.: Mahler Selecta):571--593, 2019.
\newblock Reprint of [ 0074460].

\bibitem{Moshchevitin_2012_Exponents+for+three+dimensional+simultaneous+Diophantine+approximations_MR2899740}
Nikolay Moshchevitin.
\newblock Exponents for three-dimensional simultaneous {D}iophantine approximations.
\newblock {\em Czechoslovak Math. J.}, 62(137)(1):127--137, 2012.

\bibitem{Roy_2015_On_Schmidt_and_Summerer_parametric_goemMR3418530}
Damien Roy.
\newblock On {S}chmidt and {S}ummerer parametric geometry of numbers.
\newblock {\em Ann. of Math. (2)}, 182(2):739--786, 2015.

\bibitem{Roy_2016_Spectrum_of_the_exponents_of_best_rational_approximation_MR3489062}
Damien Roy.
\newblock Spectrum of the exponents of best rational approximation.
\newblock {\em Math. Z.}, 283(1-2):143--155, 2016.

\bibitem{Schmidt_2020_On_parametric_geometry_of_numbers_MR4121876}
Wolfgang~M. Schmidt.
\newblock On parametric geometry of numbers.
\newblock {\em Acta Arith.}, 195(4):383--414, 2020.

\bibitem{Schmidt_Summerer_2009_Parametric_geometry_of_numbers_and_applications_MR2557854}
Wolfgang~M. Schmidt and Leonhard Summerer.
\newblock Parametric geometry of numbers and applications.
\newblock {\em Acta Arith.}, 140(1):67--91, 2009.

\bibitem{Schmidt_Summerer_2013_Diophantine_approximation_and_parametric_geometry_of_numbers_MR3016519}
Wolfgang~M. Schmidt and Leonhard Summerer.
\newblock Diophantine approximation and parametric geometry of numbers.
\newblock {\em Monatsh. Math.}, 169(1):51--104, 2013.

\bibitem{Solan_2021_Parametric+geometry+of+numbers+with+general+flow}
Omri~Nisan Solan.
\newblock Parametric geometry of numbers with general flow.
\newblock {\em arXiv:2106.01707v2}, 2021.

\bibitem{Summerer_2019_Generalized+simultaneous+approximation_MR3990805}
Leonhard Summerer.
\newblock Generalized simultaneous approximation to {$m$} linearly dependent reals.
\newblock {\em Mosc. J. Comb. Number Theory}, 8(3):219--228, 2019.

\end{thebibliography}
\bibliographystyle{plain}

\end{document}